\documentclass[a4paper]{amsart}
\usepackage{amsthm,amsfonts,amsmath,amssymb}
\usepackage[abs]{overpic}
\usepackage[all]{xy} 
\usepackage{subfigmat} 

\newtheorem{theorem}{Theorem}[section]
\newtheorem{lemma}[theorem]{Lemma}

\newtheorem{proposition}[theorem]{Proposition}
\newtheorem{corollary}[theorem]{Corollary}

\theoremstyle{definition}
\newtheorem{definition}[theorem]{Definition}

\theoremstyle{remark}
\newtheorem{remark}[theorem]{Remark}

\newtheorem{claim}[theorem]{Claim}

\newcommand{\omu}{\overline{\mu}}
\newcommand{\e}{\varepsilon}

\makeatletter

\@addtoreset{figure}{section}
\makeatother

\makeatletter
  
  \@addtoreset{equation}{section}
\makeatother

\begin{document}
\title[Milnor invariants, $2n$-moves and $V^{n}$-moves for welded string links]{Milnor invariants, $2n$-moves and $V^{n}$-moves \\ for welded string links}

\author[Haruko A. Miyazawa]{Haruko A. Miyazawa}
\address{Institute for Mathematics and Computer Science, Tsuda University,
2-1-1 Tsuda-Machi, Kodaira, Tokyo 187-8577, Japan}
\curraddr{}
\email{aida@tsuda.ac.jp}
\thanks{}

\author[Kodai Wada]{Kodai Wada}
\address{Faculty of Education and Integrated Arts and Sciences, Waseda University, 1-6-1 Nishi-Waseda, Shinjuku-ku, Tokyo 169-8050, Japan}
\curraddr{}
\email{k.wada8@kurenai.waseda.jp}
\thanks{The second author was supported by a Grant-in-Aid for JSPS Research Fellow (\#17J08186) of the Japan Society for the Promotion of Science.}

\author[Akira Yasuhara]{Akira Yasuhara}
\address{Faculty of Commerce, Waseda University, 1-6-1 Nishi-Waseda, Shinjuku-ku, Tokyo 169-8050, Japan}
\curraddr{}
\email{yasuhara@waseda.jp}
\thanks{The third author was partially supported by a Grant-in-Aid for Scientific Research (C) (\#17K05264) of the Japan Society for the Promotion of Science and a Waseda University Grant for Special Research Projects (\#2018S-077).}

\subjclass[2010]{57M25, 57M27}

\keywords{Welded string links, Milnor invariants, $2n$-moves, $V^{n}$-moves, self-crossing virtualization, arrow calculus.}




\begin{abstract}
In a previous paper, the authors proved that Milnor link-homotopy invariants modulo~$n$ classify classical string links up to $2n$-move and link-homotopy. 
As analogues to the welded case, in terms of Milnor invariants, 
we give here two classifications of welded string links up to $2n$-move and  self-crossing virtualization, and up to $V^{n}$-move and self-crossing virtualization, respectively.  
\end{abstract}

\maketitle

\section{Introduction}\label{sec-intro}
In~\cite{M54,M57} J.~Milnor defined a family of classical link invariants, known as {\em Milnor $\omu$-invariants}. 
Given an $m$-component classical link $L$, 
Milnor invariants $\omu_{L}(I)$ are indexed by a finite sequence of elements in $\{1,\ldots,m\}$. 
In~\cite{HL} N.~Habegger and X.-S.~Lin introduced the notion of {\em classical string links} and defined Milnor invariants for classical string links. 
These invariants are called {\em Milnor $\mu$-invariants}. 
It is remarkable that $\mu$-invariants for non-repeated sequences classify classical string links up to link-homotopy~\cite{HL} 
(whereas $\omu$-invariants are not enough strong to classify classical  links with four or more components up to link-homotopy~\cite{L}). 
Here the {\em link-homotopy} is the equivalence relation on classical (string) links generated by the self-crossing change and ambient isotopy~\cite{M54}. 

A {\em $2n$-move} is a local move as illustrated in Figure~\ref{2n-move}. 
The $2n$-moves were probably first studied by S. Kinoshita in 1957~\cite{K57}. 
Since then, $2n$-moves have been well-studied in Knot Theory. 
In particular, the connections between $2n$-moves and classical link invariants are increasingly well-understood; 
see for example~\cite{K80,P,DP02,DP04}.
Recently, the authors~\cite{MWY19} established the relation between Milnor link-homotopy invariants and $2n$-moves as follows.

\begin{figure}[htbp]
  \begin{center}
    \begin{overpic}[width=11cm]{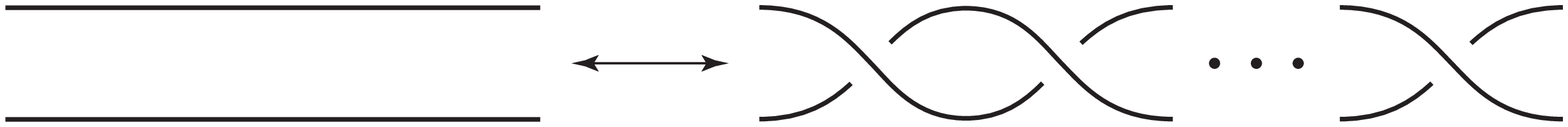}
      \put(172,-2){{\footnotesize $1$}}
      \put(210.5,-2){{\footnotesize $2$}}
      \put(287,-2){{\footnotesize $2n$}}
    \end{overpic}
  \end{center}
  \caption{$2n$-move}
  \label{2n-move}
\end{figure}

\begin{theorem}[{\cite[Theorem 1.1]{MWY19}}]
\label{th-classical}
Let $n$ be a positive integer. 
Two classical string links $\sigma$ and $\sigma'$ are $(2n+{\rm lh})$-equivalent if and only if 
$\mu_{\sigma}(I)\equiv\mu_{\sigma'}(I)\pmod{n}$ for any non-repeated sequence~$I$. 
\end{theorem}

\noindent
Here, the {\em $(2n+{\rm lh})$-equivalence} is the equivalence relation generated by the $2n$-move, self-crossing change and ambient isotopy. 

The set $\mathcal{SL}(m)$ of $m$-component classical string links has a monoid structure under the {\em stacking product}. 
Habegger and Lin proved that the set of link-homotopy classes of $\mathcal{SL}(m)$ forms a torsion free group of rank $s_{m}=\sum_{r=2}^{m}(r-2)!\binom{m}{r}$~\cite[Section 3]{HL}.  
We remark that the quotient $\mathcal{SL}(m)/(2n+\rm{lh})$ of $\mathcal{SL}(m)$ under $(2n+{\rm lh})$-equivalence forms a finite group generated by elements of order~$n$, and that the order of the group is $n^{s_{m}}$~\cite[Corollary 1.2]{MWY19}. 

The notion of {\em welded string links} was introduced by R.~Fenn, R.~Rim\'anyi and C.~Rourke~\cite{FRR}. 
M.~Goussarov, M.~Polyak and O.~Viro essentially proved that two
classical string links are equivalent as welded objects if and only if they are equivalent as classical objects~\cite[Theorem 1.B]{GPV}. 
Therefore, welded string links can be viewed as a natural extension of classical string links. 
The study of welded string links has recently become an area of active interest; see for example~\cite{BD16,BD17,ABMW-JMSJ,ABMW-Pisa,ABMW-MMJ,MY}. 

In~\cite{ABMW-Pisa} B.~Audoux, P.~Bellingeri, J.-B.~Meilhan and E.~Wagner defined Milnor invariants, 
denoted by $\mu^{{\rm w}}$, for welded string links and proved that the $\mu^{{\rm w}}$-invariants for non-repeated sequences classify welded string links up to {\rm sv}-equivalence. 
Here, the {\em {\rm sv}-equivalence} is the equivalence relation generated by the self-crossing virtualization and welded isotopy. 
A {\em crossing virtualization} is a local move replacing a classical crossing with virtual one, and a {\em self-crossing virtualization} is a crossing virtualization involving two strands of a single component. 
The {\rm sv}-equivalence is indeed a natural extension of link-homotopy in the sense that two classical string links are {\rm sv}-equivalent if and only if they are link-homotopic~\cite[Theorem 4.3]{ABMW-JMSJ}. 

As analogues of Theorem~\ref{th-classical} to the welded case, we show the following two theorems. 

\begin{theorem}\label{th-2nsv} 
Let $n$ be a positive integer. 
Two $m$-component welded string links $\sigma$ and $\sigma'$ are $(2n+{\rm sv})$-equivalent if and only if 
$\mu^{{\rm w}}_{\sigma}(I)\equiv\mu^{{\rm w}}_{\sigma'}(I)\pmod{n}$ for any non-repeated sequence $I$, and 
$\mu^{{\rm w}}_{\sigma}(ij)-\mu^{{\rm w}}_{\sigma}(ji)=\mu^{{\rm w}}_{\sigma'}(ij)-\mu^{{\rm w}}_{\sigma'}(ji)$ for any $1\leq i<j\leq m$. 
\end{theorem}

\noindent
Here, the {\em $(2n+{\rm sv})$-equivalence} is the equivalence relation on welded string links generated by the $2n$-move, self-crossing virtualization and welded isotopy.

\begin{remark}
In~\cite[Proposition 3.8]{ABMW-MMJ}, 
Audoux, Bellingeri, Meilhan and Wagner proved Theorem~\ref{th-2nsv} for $n=1$. 
Note that $\mu^{{\rm w}}(ij)-\mu^{{\rm w}}(ji)$ is written as ${\rm vlk}_{ij}-{\rm vlk}_{ji}$ in~\cite{ABMW-MMJ}. 
\end{remark}

\begin{theorem}\label{th-Vnsv}
Let $n$ be a positive integer. 
Two welded string links $\sigma$ and $\sigma'$ are $(V^{n}+{\rm sv})$-equivalent if and only if 
$\mu^{{\rm w}}_{\sigma}(I)\equiv\mu^{{\rm w}}_{\sigma'}(I)\pmod{n}$ for any non-repeated sequence $I$. 
\end{theorem}

\noindent
Here a {\em $V^{n}$-move}, defined by the authors in~\cite{MWY18}, is an oriented local move as illustrated in Figure~\ref{GV} which is a generalization of the crossing virtualization. 
(In fact, the $V^{1}$-move is equivalent to the crossing  virtualization.) 
The {\em $(V^{n}+{\rm sv})$-equivalence} is the equivalence relation on welded string links generated by the $V^{n}$-move, self-crossing virtualization and welded isotopy.

\begin{figure}[htbp]
  \begin{center}
    \begin{overpic}[width=11cm]{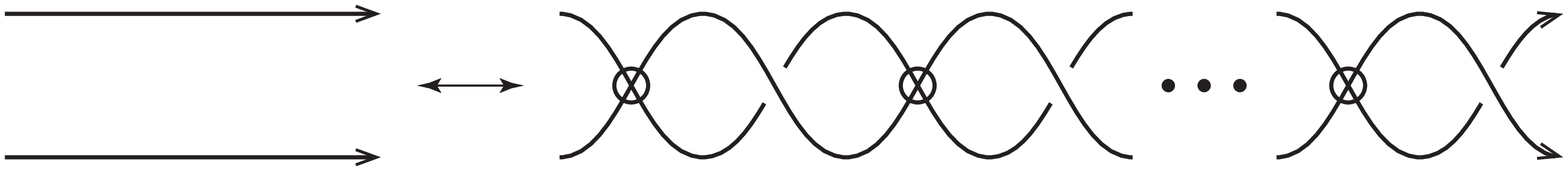}
      \put(153,2){{\footnotesize $1$}}
      \put(210,2){{\footnotesize $2$}}
      \put(296,2){{\footnotesize $n$}}
    \end{overpic}
  \end{center}
  \caption{$V^{n}$-move}
  \label{GV}
\end{figure}

Let ${\rm w}\mathcal{SL}(m)/(2n+\rm{sv})$ and ${\rm w}\mathcal{SL}(m)/(V^{n}+\rm{sv})$ denote the quotients of the set ${\rm w}\mathcal{SL}(m)$ of $m$-component welded string links under $(2n+\rm{sv})$-equivalence and $(V^{n}+\rm{sv})$-equivalence, respectively. 
Since the set of {\rm sv}-equivalence classes of ${\rm w}\mathcal{SL}(m)$ forms a group of rank~$w_{m}=\sum_{r=2}^{m}(r-2)!r\binom{m}{r}$~\cite[Remark 4.9]{AMW}, 
it is not hard to see that both  
${\rm w}\mathcal{SL}(m)/(2n+\rm{sv})$ and ${\rm w}\mathcal{SL}(m)/(V^{n}+\rm{sv})$ also form groups. 
In contrast to $\mathcal{SL}(m)/(2n+\rm{lh})$, the quotient 
${\rm w}\mathcal{SL}(m)/(2n+\rm{sv})$ is not a finite group. 
On the other hand, we have the following.

\begin{corollary}\label{cor-group} 
The quotient ${\rm w}\mathcal{SL}(m)/(V^{n}+\rm{sv})$ forms a finite group generated by elements of order~$n$, 
and the order of the group is $n^{w_{m}}$. 
\end{corollary}

Theorems~\ref{th-2nsv} and~\ref{th-Vnsv} indicate that $(2n+{\rm sv})$-equivalence implies $(V^{n}+{\rm sv})$-equivalence (Proposition~\ref{prop-2nVn}). 
Moreover, these theorems together with Theorem~\ref{th-classical} imply 
that both natural maps $\iota_{1}:\mathcal{SL}(m)/(2n+{\rm lh})\rightarrow{\rm w}\mathcal{SL}(m)/(2n+{\rm sv})$ and 
$\iota_{2}:\mathcal{SL}(m)/(2n+{\rm lh})\rightarrow{\rm w}\mathcal{SL}(m)/(V^{n}+{\rm sv})$ are injective 
(Proposition~\ref{prop-injective}). 
Figure~\ref{summary} gives the summary of relations between $\mathcal{SL}(m)/(2n+{\rm lh})$, ${\rm w}\mathcal{SL}(m)/(2n+{\rm sv})$ and ${\rm w}\mathcal{SL}(m)/(2n+{\rm sv})$.

\begin{figure}[htbp]
  \begin{subfigmatrix}{1}{
    \[
    \xymatrix
    {
      \mathcal{SL}(m)/(2n+{\rm lh})
      \ar@{^{(}->}[r]^-{\iota_{1}}
      \ar@<-0.3ex>@{^{(}->}[rd]_-{\iota_{2}}
      &{\rm w}\mathcal{SL}(m)/(2n+{\rm sv})
      \ar@{->>}[d]^-{\text{Proposition~\ref{prop-2nVn}}}\\
      &{\rm w}\mathcal{SL}(m)/(V^{n}+{\rm sv})
    }
    \]
  }
  \end{subfigmatrix}
  \caption{Relations between $\mathcal{SL}(m)/(2n+{\rm lh})$, ${\rm w}\mathcal{SL}(m)/(2n+{\rm sv})$ and ${\rm w}\mathcal{SL}(m)/(2n+{\rm sv})$}
  \label{summary}
\end{figure}
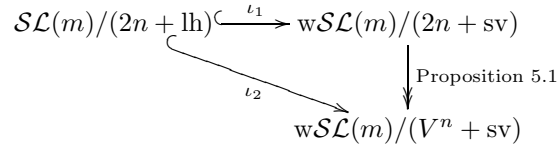

\section{Welded string links and welded Milnor invariants}
In this section, we review the definitions of welded string links and their Milnor invariants.

\subsection{Welded string links}
An {\em $m$-component virtual string link diagram} is an $m$-component string link diagram in the plane, whose transverse double points admit not only {\em classical crossings} but also {\em virtual crossings} illustrated in Figure~\ref{xing}. 
Throughout the paper, virtual string link diagrams are assumed to be ordered and oriented.

\begin{figure}[htbp]
  \begin{center}
    \begin{overpic}[width=4.5cm]{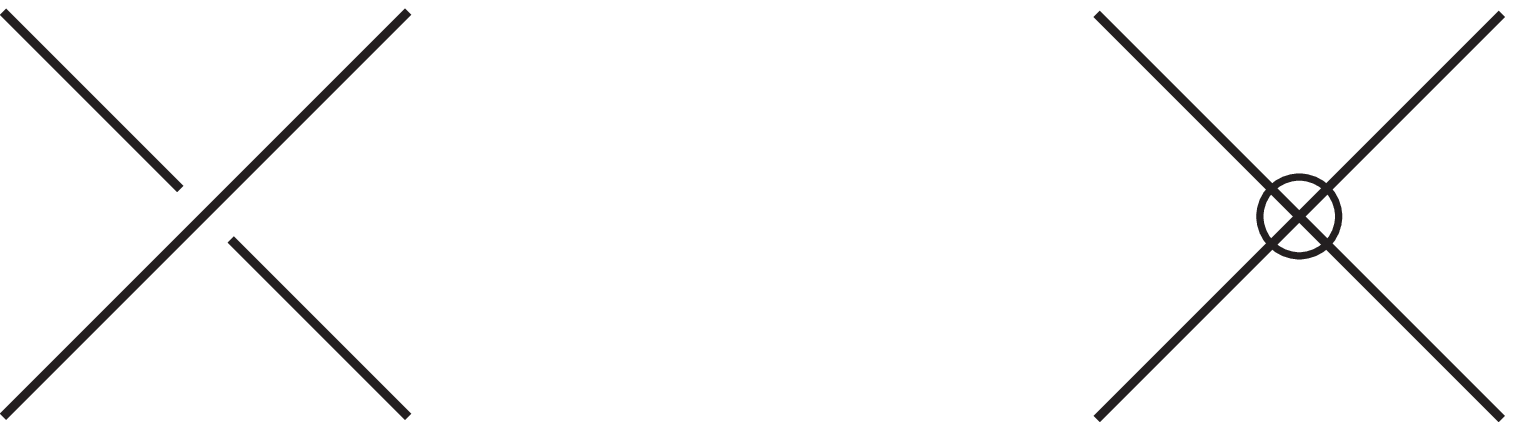}
      \put(-19,-15){classical crossing}
      \put(78,-15){virtual crossing}
    \end{overpic}
  \end{center}
  \vspace{1em}
  \caption{}
  \label{xing}
\end{figure}

A {\em welded string link} is an equivalence class of virtual string link diagrams under {\em welded Reidemeister moves}, which consist of Reidemeister moves R1--R3, virtual moves V1--V4 and the overcrossings commute move OC illustrated in Figure~\ref{wRmoves}. 
A sequence of welded Reidemeister moves is called a {\em welded isotopy}.

\begin{figure}[htbp]
  \begin{center}
    \begin{overpic}[width=12cm]{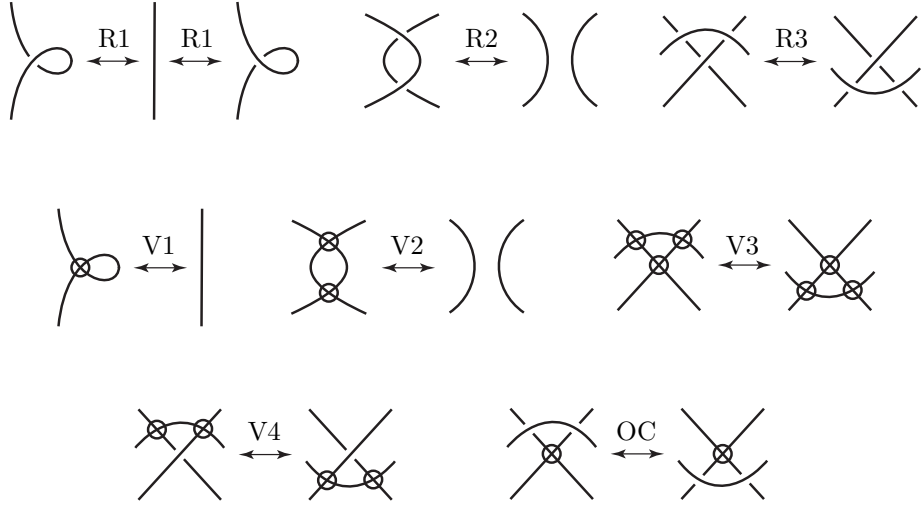}
      \put(33,171){R1}
      \put(64,171){R1}
      \put(170.5,171){R2}
      \put(286.5,171){R3}
      \put(49.5,93){V1}
      \put(142.5,93){V2} 
      \put(268,93){V3}
      \put(89,23){V4}
      \put(227,23){OC}
    \end{overpic}
  \end{center}
  \caption{Welded Reidemeister moves}
  \label{wRmoves}
\end{figure}

Here, we give the definition of the group of a welded string link. 
The {\em group $G(D)$} of a virtual string link diagram $D$ is defined via the Wirtinger presentation~\cite[Section 4]{Kauffman},  
i.e. an arc of $D$ yields a generator, and each classical crossing gives a relation of the form $c^{-1}b^{-1}ab$, 
where $a$ and $c$ correspond to the underpasses and $b$ corresponds to the overpass at the crossing; see Figure~\ref{relation}. 
(Here, an {\em arc} of $D$ is a piece of strand such that each boundary is either a strand endpoint or a classical undercrossing, and the interior does not contain classical undercrossings.) 
The group $G(D)$ is preserved under welded isotopy~\cite{Kauffman, Satoh}, 
and hence we define the {\em group $G(\sigma)$} of a welded string link $\sigma$ to be $G(D)$ of a virtual diagram $D$ of $\sigma$.

\begin{figure}[t]
  \begin{center}
    \begin{overpic}[width=1.5cm]{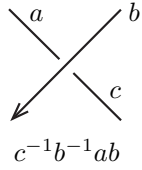}
      \put(8,38){$a$}
      \put(45,38){$b$}
      \put(38,9){$c$}
      \put(2,-15){$c^{-1}b^{-1}ab$}
    \end{overpic}
  \end{center}
  \vspace{1em}
  \caption{Wirtinger relation}
  \label{relation}
\end{figure}

\subsection{Welded Milnor invariants}\label{subsec-wMilnor}
In~\cite{ABMW-Pisa}, Audoux, Bellingeri, Meilhan and Wagner defined Milnor invariants for {\em ribbon $2$-dimensional string links}, i.e. properly embedded annuli in the $4$-ball bounding immersed $3$-balls with only ribbon singularities. 
They also defined a welded extension of Milnor invariants, which is an invariant of welded string links, via the {\em Tube map} (see~\cite{Yajima,Satoh}) sending welded string links to ribbon $2$-dimensional string links. 
These invariants are called {\em welded Milnor invariants}. 
The construction of welded Milnor invariants is topological, since it is defined via the Tube map. 
However, applying Milnor's algorithm given in~\cite{M57}, we can (define and) compute welded Milnor invariants by means of virtual diagrams as follows. 

Given an $m$-component welded string link $\sigma$, 
consider its virtual diagram $D_{1}\cup\cdots\cup D_{m}$. 
Put labels $a_{i1},a_{i2},\ldots,a_{ir(i)}$ in order on all arcs of the $i$th component $D_{i}$ 
while we go along orientation on $D_{i}$ from the initial arc, 
where $r(i)$ denotes the number of arcs of $D_{i}$ $(i=1,\ldots,m)$. 
We call the arc $a_{i1}$ of $D_{i}$ the {\em $i$th meridian}.  
The Wirtinger presentation of $G(\sigma)$ has the form 
\[
\left\langle a_{ij}\ (1\leq i\leq m,1\leq j\leq r(i))~\vline~a_{ij+1}^{-1}u_{ij}^{-1}a_{ij}u_{ij}\ (1\leq i\leq m,1\leq j\leq r(i)-1)\right\rangle, 
\]
where the $u_{ij}$ are generators or inverses of generators that depend on the signs of the classical crossings. 
Here we set 
\[
v_{ij}=u_{i1}u_{i2}\ldots u_{ij}. 
\]
We call the product $v_{ir(i)-1}$ an {\em $i$th longitude}. 
Furthermore, we obtain the {\em preferred longitude $l_{i}$} by multiplying $v_{ir(i-1)}$ by $a_{i1}^{s}$ on the left for some $s\in\mathbb{Z}$. 

Let $G(\sigma)_{q}$ denote the $q$th term of the lower central series of $G(\sigma)$, 
and let $\alpha_{i}$ denote the image of $a_{i1}$ in the quotient $G(\sigma)/G(\sigma)_{q}$. 
Since $G(\sigma)/G(\sigma)_{q}$ is generated by $\alpha_{1},\ldots,\alpha_{m}$ (see~\cite{C,ABMW-Pisa}), 
the $i$th preferred longitude $l_{i}$ is expressed modulo $G(\sigma)_{q}$ as a word in $\alpha_{1},\ldots,\alpha_{m}$ for each $i\in\{1,\ldots,m\}$. 
We denote this word by $\lambda_{i}$. 

Let $\langle\alpha_{1},\ldots,\alpha_{m}\rangle$ denote the free group on $\{\alpha_{1},\ldots,\alpha_{m}\}$, and 
let $\mathbb{Z}\langle\langle X_{1},\ldots,X_{m}\rangle\rangle$ denote the ring of formal power series in noncommutative variables $X_{1},\ldots,X_{m}$ with integer coefficients. 
The {\em Magnus expansion} is a homomorphism 
\[
E:\langle\alpha_{1},\ldots,\alpha_{m}\rangle
\longrightarrow \mathbb{Z}\langle\langle X_{1},\ldots,X_{m}\rangle\rangle
\]
defined by, for $1\leq i\leq m$, 
\[
E(\alpha_{i})=1+X_{i}, \ E(\alpha_{i}^{-1})=1-X_{i}+X_{i}^{2}-X_{i}^{3}+\cdots. 
\]

\begin{definition}
For a sequence $I=j_{1}j_{2}\ldots j_{k}i$ $(k<q)$ of elements in $\{1,\ldots,m\}$,  
the {\em welded Milnor invariant $\mu^{{\rm w}}_{\sigma}(I)$} of $\sigma$ is the coefficient of $X_{j_{1}}\cdots X_{j_{k}}$ in the Magnus expansion $E(\lambda_{i})$. 
\end{definition}

\begin{remark}[{\cite[Theorem 5.4]{ABMW-Pisa}}]\label{rem-weldedMilnor}
The $\mu^{{\rm w}}$-invariant is indeed a welded extension of the (classical) Milnor $\mu$-invariant in the sense that if $\sigma$ is a classical string link, then $\mu^{{\rm w}}_{\sigma}(I)=\mu_{\sigma}(I)$ for any sequence $I$. 
\end{remark}

To compute $\mu^{{\rm w}}_{\sigma}(I)$, we need to obtain the word $\lambda_{i}$ in $\alpha_{1},\ldots,\alpha_{m}$ concretely. 
In~\cite{M57}, Milnor introduced an algorithm to give $\lambda_{i}$ by using the Wirtinger presentation of $G(\sigma)$ and a sequence of homomorphisms $\eta_{q}$. 
(Although this algorithm was actually given for Milnor invariants of links, it can be applied to those of welded string links.)

Let $\overline{A}$ denote the free group on the Wirtinger generators $\{a_{ij}\}$, 
and let $A$ denote the free subgroup generated by $a_{11},a_{21},\ldots,a_{m1}$. 
A sequence of homomorphisms $\eta_{q}:\overline{A}\rightarrow A$ is defined inductively by 
\begin{eqnarray*}
\eta_{1}(a_{ij})=a_{i1},& \\
\eta_{q+1}(a_{i1})=a_{i1},& 
\eta_{q+1}(a_{ij+1})=\eta_{q}(v_{ij}^{-1}a_{i1}v_{ij}). 
\end{eqnarray*}
Let $\overline{A}_{q}$ denote the $q$th term of the lower central series of $\overline{A}$, and 
let $N$ denote the normal subgroup of $\overline{A}$ generated by the Wirtinger relations $\{a_{ij+1}^{-1}u_{ij}^{-1}a_{ij}u_{ij}\}$. 
Milnor proved that 
\[
\eta_{q}(a_{ij})\equiv a_{ij}\pmod{\overline{A}_{q}N}.  
\]
Hence, by the congruence above,
we can identify $\phi\circ\eta_{q}(l_{i})$ with $\lambda_{i}$, 
where $\phi:A\rightarrow\langle\alpha_{1},\ldots,\alpha_{m}\rangle$ is a homomorphism defined by $\phi(a_{i1})=\alpha_{i}$ $(i=1,\ldots,m)$.

\section{Welded Milnor invariants and $V^{n}$-moves} 
In this section, we discuss the invariance of welded Milnor invariants under $V^{n}$-moves. 
We start with the following theorem concerning $\mu^{{\rm w}}$-invariants for non-repeated sequences.

\begin{theorem}\label{th-inv-Vn}
Let $n$ be a positive integer. 
If two welded string links $\sigma$ and $\sigma'$ are $(V^{n}+{\rm sv})$-equivalent, then $\mu^{{\rm w}}_{\sigma}(I)\equiv\mu^{{\rm w}}_{\sigma'}(I)\pmod{n}$ for any non-repeated sequence~$I$.
\end{theorem}

\begin{proof}
It is obvious for $n=1$, and hence we consider the case $n\geq2$. 
Since $\mu^{{\rm w}}$-invariants for non-repeated sequences are {\rm sv}-equivalence invariants, 
we show that their residue classes modulo $n$ are preserved under $V^{n}$-moves. 

Let $D$ and $D'$ be virtual diagrams of $m$-component welded string links $\sigma$ and $\sigma'$, respectively. 
Assume that $D$ and $D'$ are related by a single $V^{n}$-move in a disk~$\Delta$; see Figure~\ref{Vn-move}.  
Since a $V^{n}$-move involving two strands of a single component is realized by {\rm sv}-equivalence,  
we may assume that two strands in the disk $\Delta$ belong to different components. 
Put labels $a_{ij}$ $(1\leq i\leq m$, $1\leq j\leq r(i))$ on all arcs of $D$ as described in Section~\ref{subsec-wMilnor}, 
and put labels $a'_{ij}$ on all arcs in $D'\setminus\Delta$ 
which correspond to the arcs labeled $a_{ij}$ in $D\setminus\Delta$. 
Also put labels $b'_{1},\ldots,b'_{n}$ on the arcs of $D'$ in $\Delta$ as illustrated in Figure~\ref{Vn-move}. 
Let $\overline{A'}$ be the free group on $\{a'_{ij}\}\cup\{b'_{1},\ldots,b'_{n}\}$ 
and $A'$ the free subgroup on $\{a'_{11},a'_{21},\ldots,a'_{m1}\}$. 
Let $\eta'_{q}:\overline{A'}\rightarrow A'$ denote the sequence of homomorphisms associated with $D'$ given in Section~\ref{subsec-wMilnor}, 
and define a homomorphism $\phi':A'\rightarrow\langle\alpha_{1},\ldots,\alpha_{m}\rangle$ by $\phi'(a'_{i1})=\alpha_{i}$ $(i=1,\ldots,m)$.

\begin{figure}[htbp]
  \begin{center}
    \vspace{1em}
    \begin{overpic}[width=11cm]{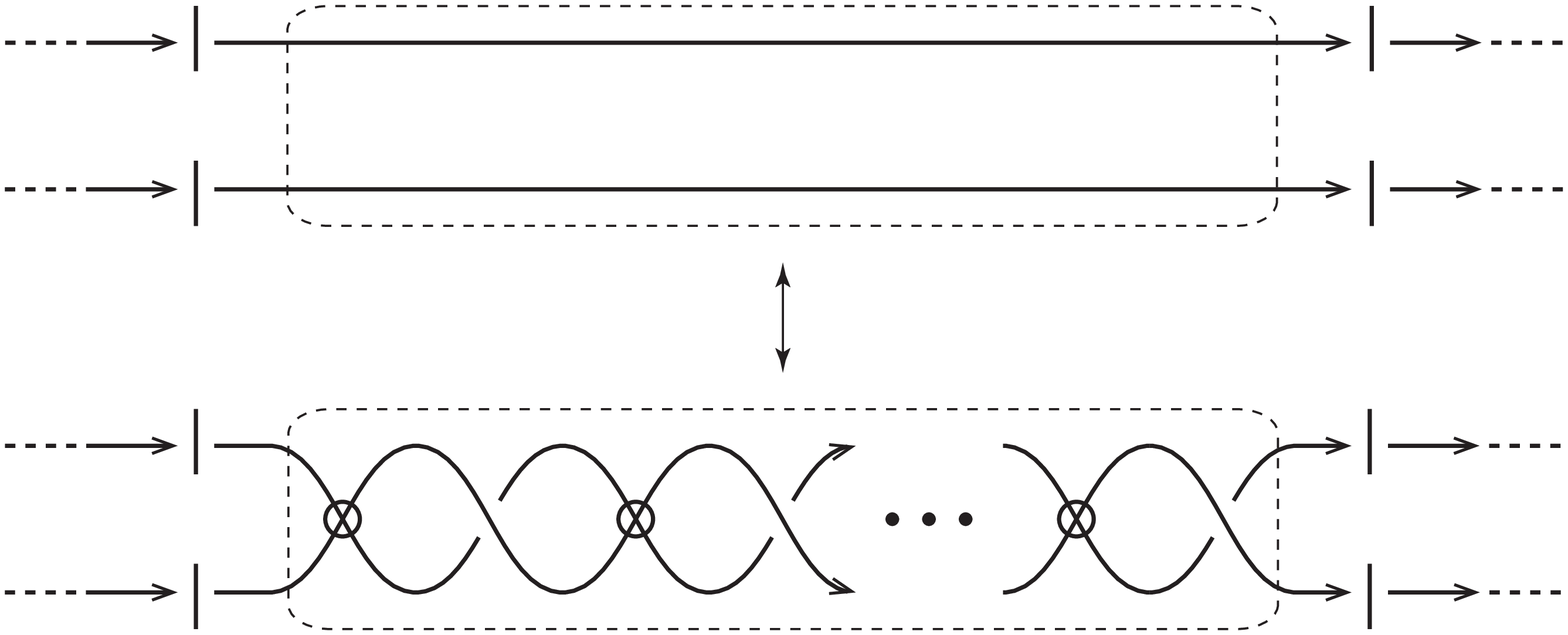}
      \put(-20,100){$D$ :}
      \put(-20,20){$D'$ :}
      \put(162,59){$V^{n}$-move}
      \put(152,129){$\Delta$}
      \put(152,-10){$\Delta$}
      \put(11,124){$a_{gh-1}$}
      \put(11,79){$a_{kl-1}$}
      \put(43,124){$a_{gh}$}
      \put(43,79){$a_{kl}$}
      \put(277,124){$a_{gh+1}$}
      \put(277,79){$a_{kl+1}$}
      \put(11,44){$a'_{gh-1}$}
      \put(11,-3){$a'_{kl-1}$}
      \put(43,44){$a'_{gh}$}
      \put(43,-3){$a'_{kl}$}
      \put(277,44){$a'_{gh+1}$}
      \put(277,-3){$a'_{kl+1}$}
      \put(94.5,34){$b'_{1}$}
      \put(152,34){$b'_{2}$}
      \put(239,34){$b'_{n}$}
    \end{overpic}
  \end{center}
  \caption{$D$ and $D'$ are related by a single $V^{n}$-move.}
  \label{Vn-move}
\end{figure}

Here, for $P,Q\in\mathbb{Z}\langle\langle X_{1},\cdots,X_{m}\rangle\rangle$, we use the notation $P\overset{(n)}{\equiv}Q$ if $P-Q$ is contained in the ideal generated by $n$. 
For the $i$th preferred longitudes $l_{i}$ and $l'_{i}$ associated with $D$ and $D'$, respectively, it is enough to show that 
\begin{equation}\label{eq-Vn}
E\left(\phi\circ\eta_{q}\left(l_{i}\right)\right)
\overset{(n)}{\equiv}E\left(\phi'\circ\eta'_{q}\left(l'_{i}\right)\right)+\mathcal{O}(2) 
\end{equation}
for any $1\leq i\leq m$, 
where $\mathcal{O}(2)$ denotes $0$ or the terms containing $X_{r}$ at least two for some $r$ $(=1,\ldots,m)$.  
Without loss of generality we may assume that $i=1$, i.e. 
we compare $l_{1}=a_{11}^{s}v_{1r(1)-1}$ and $l'_{1}={a'}^{t}_{11}v'_{1r(1)-1}$ $(s,t\in\mathbb{Z})$. 
Recall that two strands in $\Delta$ belong to different components. 
This implies that $s=t$. 

If $g\neq1$ in Figure~\ref{Vn-move}, then $l'_{1}$ is obtained from $l_{1}$ by replacing $u_{1j}$ with $u'_{1j}$ $(j=1,\ldots,r(1)-1)$ and $a_{11}$ with $a'_{11}$. 

If $g=1$ in Figure~\ref{Vn-move}, then  
$l_{1}$ and $l'_{1}$ can be written respectively in the forms 
\[
l_{1}=a_{11}^{s}u_{11}\ldots u_{1h-1}u_{1h}\ldots u_{1r(1)-1}
\]
and
\[
l'_{1}={a'}^{s}_{11}u'_{11}\ldots u'_{1h-1}{a'}^{n}_{kl}u'_{1h}\ldots u'_{1r(1)-1}.
\]
Therefore, in both cases, 
Congruence~(\ref{eq-Vn}) follows from the claim below. 
\end{proof}

\begin{claim}\label{claim-Vn}
Let $n\geq2$ be an integer and $\e\in\{1,-1\}$. 
For any $1\leq i\leq m$ and $1\leq j\leq r(i)$, 
the following $(1)$ and $(2)$ hold: 
\begin{enumerate}
\item 
$E\left(\phi'\circ\eta'_{q}\left({a'}^{\e n}_{ij}\right)\right)
\overset{(n)}{\equiv}
1+\mathcal{O}(2)$. 

\item 
$E\left(\phi\circ\eta_{q}\left(a_{ij}\right)\right)
\overset{(n)}{\equiv}
E\left(\phi'\circ\eta'_{q}\left({a'}_{ij}\right)\right)+\mathcal{O}(2)$. 
\end{enumerate}
\end{claim}

\begin{proof}
By the definition of $\eta'_{q}$, 
it follows that 
$\phi'\circ\eta'_{q}\left({a'}_{ij}^{\e}\right)=w^{-1}\alpha_{i}^{\e}w$ for some word $w$ in $\alpha_{1},\ldots,\alpha_{m}$. 
Set $E(w)=1+W$ and $E(w^{-1})=1+\overline{W}$, 
where $W$ and $\overline{W}$ denote the terms of degree $\geq1$ such that $\left(1+\overline{W}\right)\left(1+W\right)=1$. 
Then we have  
\begin{eqnarray*}
E\left(\phi'\circ\eta'_{q}\left({a'}_{ij}^{\e}\right)\right) 
&=&E\left(w^{-1}\alpha_{i}^{\e}w\right) \\ 
&=&\left(1+\overline{W}\right)\left(1+\e X_{i}\right)\left(1+W\right)+\mathcal{O}(2) \\
&=&1+\e X_{i}+\e X_{i}W+\e\overline{W}X_{i}+\e\overline{W}X_{i}W+\mathcal{O}(2) \\
&=&1+\e P(X_{i})+\mathcal{O}(2), 
\end{eqnarray*}
where $P(X_{i})=X_{i}+X_{i}W+\overline{W}X_{i}+\overline{W}X_{i}W$. 
Note that each term in $P(X_{i})$ contains $X_{i}$.
Therefore, it follows that 
\[
E\left(\phi'\circ\eta'_{q}\left({a'}^{\e n}_{ij}\right)\right)
=\left(1+\e P(X_{i})+\mathcal{O}(2)\right)^{n}
=1+\e nP(X_{i})+\mathcal{O}(2)
\overset{(n)}{\equiv}1+\mathcal{O}(2). 
\]
This completes the proof of Claim~\ref{claim-Vn}~(1).

The proof of Claim~\ref{claim-Vn}~(2) is done by induction on $q$. 
The assertion certainly holds for $q=1$. 
Recall that 
\[
\phi\circ\eta_{q+1}\left(a_{ij+1}\right)=\phi\circ\eta_{q}\left(v_{ij}^{-1}a_{i1}v_{ij}\right)\]
and
\[\phi'\circ\eta'_{q+1}\left(a'_{ij+1}\right)=\phi'\circ\eta'_{q}\left({v'_{ij}}^{-1}a'_{i1}v'_{ij}\right).
\]

If $v_{ij}$ does not pass through $\Delta$ or $g\neq i$ in Figure~\ref{Vn-move}, 
then $v'_{ij}$ is obtained from $v_{ij}$ by replacing $a_{ij}$ with $a'_{ij}$.   
Hence, $E\left(\phi\circ\eta_{q}\left(v_{ij}\right)\right)\overset{(n)}{\equiv}E\left(\phi'\circ\eta'_{q}\left(v'_{ij}\right)\right)+\mathcal{O}(2)$ by the induction hypothesis. 
This implies that  
\begin{eqnarray*}
E\left(\phi\circ\eta_{q+1}\left(a_{ij+1}\right)\right)
&=&E\left(\phi\circ\eta_{q}\left(v_{ij}^{-1}a_{i1}v_{ij}\right)\right) \\
&\overset{(n)}{\equiv}&E\left(\phi'\circ\eta'_{q}\left({v'_{ij}}^{-1}a'_{i1}v'_{ij}\right)\right)+\mathcal{O}(2) \\
&=&E\left(\phi'\circ\eta'_{q+1}\left(a'_{ij+1}\right)\right)+\mathcal{O}(2). 
\end{eqnarray*}

If $v_{ij}$ passes through $\Delta$ and $g=i$ in Figure~\ref{Vn-move}, 
then $v_{ij}$ and $v'_{ij}$ can be written respectively in the forms 
\[
v_{ij}=u_{i1}\ldots u_{ih-1}u_{ih}\ldots u_{ij}
\]
and
\[ 
v'_{ij}=u'_{i1}\ldots u'_{ih-1}{a'}_{kl}^{n}u'_{ih}\ldots u'_{ij}.
\] 
By Claim~\ref{claim-Vn}~(1) and the induction hypothesis, 
it follows that $E\left(\phi\circ\eta_{q}\left(v_{ij}\right)\right)\overset{(n)}{\equiv}E\left(\phi'\circ\eta'_{q}\left(v'_{ij}\right)\right)+\mathcal{O}(2)$. 
This completes the proof of Claim~\ref{claim-Vn}~(2). 
\end{proof}

\begin{proposition}\label{prop-inv-2n}
Let $n$ be a positive integer. 
If two $m$-component welded string links $\sigma$ and $\sigma'$ are $(2n+{\rm sv})$-equivalent, then 
$\mu^{{\rm w}}_{\sigma}(I)\equiv\mu^{{\rm w}}_{\sigma'}(I)\pmod{n}$ for any non-repeated sequence $I$, and 
$\mu^{{\rm w}}_{\sigma}(ij)-\mu^{{\rm w}}_{\sigma}(ji)=\mu^{{\rm w}}_{\sigma'}(ij)-\mu^{{\rm w}}_{\sigma'}(ji)$ for any $1\leq i<j\leq m$. 
\end{proposition}

\begin{proof}
As mentioned in Section~\ref{sec-intro}, $(2n+{\rm sv})$-equivalence implies $(V^{n}+{\rm sv})$-equivalence (Proposition~\ref{prop-2nVn}). 
This together with Theorem~\ref{th-inv-Vn} implies that the residue class of $\mu^{{\rm w}}(I)$ modulo~$n$ is preserved under $(2n+{\rm sv})$-equivalence.

By a single $2n$-move involving two strands of the $k$th and the $l$th components, 
both of the changes of $\mu^{{\rm w}}(ij)$ and $\mu^{{\rm w}}(ji)$ are $\e n$ $(\e\in\{1,-1\})$ if $\{k,l\}=\{i,j\}$ and $0$ otherwise. 
Furthermore, since $\mu^{{\rm w}}(ij)$ and $\mu^{{\rm w}}(ji)$ are {\rm sv}-equivalence invariants, 
the integer $\mu^{{\rm w}}(ij)-\mu^{{\rm w}}(ji)$ is preserved under $(2n+{\rm sv})$-equivalence. 
This completes the proof.
\end{proof}

For $\mu^{{\rm w}}$-invariants possibly with {\em repeated} sequences, we have the following.

\begin{proposition}\label{prop-inv-prime}
Let $p$ be a prime number. 
If two welded string links $\sigma$ and $\sigma'$ are related by $V^{p}$-moves, then $\mu^{{\rm w}}_{\sigma}(I)\equiv\mu^{{\rm w}}_{\sigma'}(I)\pmod{p}$ for any sequence $I$ of length~$\leq p$. 
\end{proposition}

\begin{proof}
Let $D$ and $D'$ be virtual diagrams of $m$-component welded string links $\sigma$ and $\sigma'$, respectively. 
Assume that $D$ and $D'$ are related by a single $V^{p}$-move.  
We use the same notation as in the proof of Theorem~\ref{th-inv-Vn}.  
It is enough to show that, for any $1\leq i\leq m$,  
\begin{equation*}
E\left(\phi\circ\eta_{q}\left(l_{i}\right)\right)
\overset{(p)}{\equiv}E\left(\phi'\circ\eta'_{q}\left(l'_{i}\right)\right)+(\text{terms of degree $\geq p$}).  
\end{equation*} 
By arguments similar to those in the proof of Theorem~\ref{th-inv-Vn}, 
$l'_{i}$ is obtained from $l_{i}$ by replacing $a_{kl}$ with $a'_{kl}$ for all $k,l$ and inserting the $p$th powers of elements in the free group $\overline{A'}$.  
The following claim, which was proved in~\cite{MWY19}, 
 completes the proof. 
\end{proof}

\begin{claim}[{\cite[Claim 3.6]{MWY19}}]
(1)~For any word $w$ in $\alpha_{1},\ldots,\alpha_{m}$, we have 
\[
E\left(w^{p}\right)
\overset{(p)}{\equiv}
1+(\text{terms of degree $\geq p$}). 
\]

\noindent
(2)~For any $1\leq i\leq m$ and $1\leq j\leq r(i)$, we have 
\[
E\left(\phi\circ\eta_{q}\left(a_{ij}\right)\right)
\overset{(p)}{\equiv} 
E\left(\phi'\circ\eta'_{q}\left({a'}_{ij}\right)\right)+(\text{terms of degree $\geq p$}). 
\] 
\end{claim}

\section{Arrow calculus}
To show Theorems~\ref{th-2nsv} and~\ref{th-Vnsv} 
we will use {\em arrow calculus}, introduced by Meilhan and the third author in~\cite{MY}, 
which is a welded version of the theory of claspers~\cite{H}. 
In this section, we briefly recall the basic notions of arrow calculus from~\cite{MY}.

\subsection{Definitions}
\begin{definition}
Let $D$ be a virtual string link diagram. 
An immersed connected uni-trivalent tree $T$ in the plane of the diagram is called a {\em {\rm w}-tree} for $D$ if it satisfies the following: 
\begin{enumerate}
\item The trivalent vertices of $T$ are pairwise disjoint and disjoint from $D$. 

\item The univalent vertices of $T$ are pairwise disjoint and are contained in $D\setminus\{\text{crossings of $D$}\}$. 

\item All edges are oriented such that each trivalent vertex has two ingoing and one outgoing edge. 

\item All singularities of $T$ and those between $D$ and $T$ are virtual crossings. 

\item Each edge of $T$ has a number (possibly zero) of decorations $\bullet$, called {\em twists}, which are disjoint from all vertices and crossings.  
\end{enumerate}
The univalent vertices of $T$ with outgoing edges are called {\em tails}, and the unique univalent vertex of $T$ with an ingoing edge is called the {\em head}. 
Tails and the head are also called {\em endpoints} when we do not need to distinguish between them. 
The {\em terminal edge} of $T$ is the edge which is incident to the head. 
We say that $T$ is a {\em {\rm w}-tree of degree $k$} or {\em ${\rm w}_{k}$-tree} if $T$ has $k$ tails. 
In particular, a ${\rm w}_{1}$-tree is called a {\em {\rm w}-arrow}.
\end{definition}

Given a uni-trivalent tree, 
picking a univalent vertex as the head uniquely determines an orientation on all edges respecting the above rule. 
Hence, we may only indicate the orientation on {\rm w}-trees at the terminal edge. 

For a union of {\rm w}-trees, vertices are assumed to be pairwise disjoint, and crossings among edges are assumed to be virtual. 
Hereafter, diagrams are drawn with bold lines, while {\rm w}-trees are drawn with thin lines. 
Furthermore, we do not draw small circles around virtual crossings between {\rm w}-trees and between {\rm w-trees} and diagrams, 
while we keep small circles between diagrams.

\subsection{Surgery along w-trees}
The {\rm w}-trees are equipped with surgery operations on virtual diagrams. 
This subsection gives the definition of surgery along {\rm w}-trees. 

We first consider the case of {\rm w}-arrows. 
Let $A$ be a union of w-arrows for a virtual string link diagram $D$. 
{\em Surgery along $A$} on $D$ yields a new virtual string link diagram, denoted by $D_{A}$, as follows.  
Assume that there exists a disk in the plane which intersects $D\cup A$ as illustrated in Figure~\ref{surgery}. 
Then the figure indicates the result of surgery along a {\rm w}-arrow of $A$ on $D$. 
We emphasize that the surgery operation depends on the orientation of the strand of $D$ containing the tail of the {\rm w}-arrow.

\begin{figure}[htbp]
  \begin{center}
    \begin{overpic}[width=9cm]{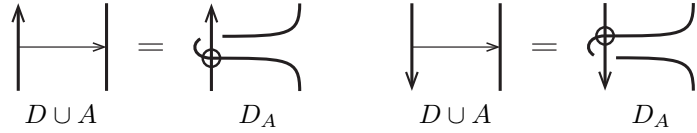}
      \put(5,-12){$D\cup A$} 
      \put(85,-12){$D_{A}$} 
      \put(152.5,-12){$D\cup A$}
      \put(232,-12){$D_{A}$} 
    \end{overpic}
  \end{center}
  \vspace{1em}
  \caption{Surgery along a {\rm w}-arrow of $A$ on $D$}
  \label{surgery}
\end{figure}

If a {\rm w}-arrow of $A$ intersects a (possibly the same) {\rm w}-arrow (resp. $D$), then the result of surgery is essentially the same as above but each intersection introduces virtual crossings illustrated in the left-hand side (resp. center) of Figure~\ref{surgery2}.
Furthermore, if a {\rm w}-arrow of $A$ has some twists, then each twist is converted to a half-twist whose crossing is virtual; see  the right-hand side of Figure~\ref{surgery2}.

\begin{figure}[htbp]
  \begin{center}
    \begin{overpic}[width=12cm]{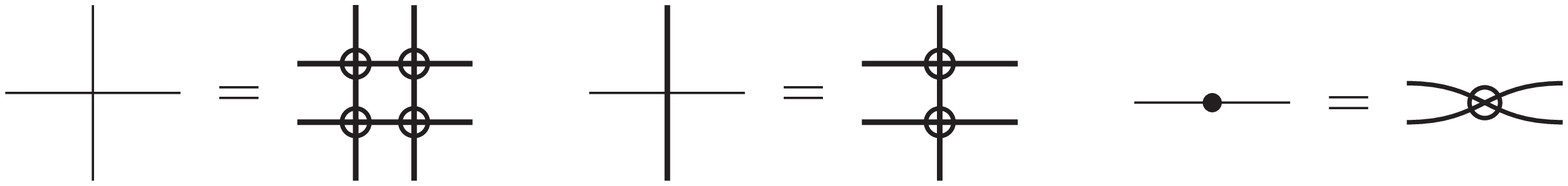}
      \put(-5,22){$A$}
      \put(14.5,-12){$A$} 
      \put(77,-12){$D_{A}$} 
      \put(122,22){$A$}
      \put(140,-12){$D$}
      \put(198,-12){$D_{A}$} 
      \put(259,-12){$A$}
      \put(317,-12){$D_{A}$} 
    \end{overpic}
  \end{center}
  \vspace{1em}
  \caption{}
  \label{surgery2}
\end{figure}

An {\em arrow presentation} for a virtual string link diagram $D$ is a pair $(V,A)$ of a virtual string link diagram $V$ without classical crossings and a union $A$ of {\rm w}-arrows for $V$ such that $V_{A}$ is welded isotopic to $D$. 
Any virtual string link diagram has an arrow presentation 
because any classical crossing can be replaced by a virtual one with a {\rm w}-arrow; see Figure~\ref{Aprst}. 
Two arrow presentations $(V,A)$ and $(V',A')$ are {\em equivalent} if $V_{A}$ and $V'_{A'}$ are welded isotopic. 
In~\cite[Section 4.3]{MY}, Meilhan and the third author gave 
a list of local moves on arrow presentations, 
which are called {\em arrow moves}.  
They proved that two arrow presentations are equivalent if and only if they are related by a sequence of arrow moves~\cite[Theorem~4.5]{MY}.

\begin{figure}[htbp]
  \begin{center}
    \begin{overpic}[width=6cm]{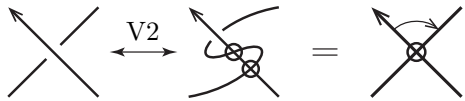}
      \put(44.5,22){V2}
    \end{overpic}
  \end{center}
  \caption{Any classical crossing can be replaced by a virtual one with a {\rm w}-arrow.}
  \label{Aprst}
\end{figure}

Now we define surgery along {\rm w}-trees. 
We start with some preliminary definitions. 
A {\em subtree} of a {\rm w}-tree is a connected union of edges and vertices of the {\rm w}-tree. 
Let $S$ be a subtree  of a {\rm w}-tree $T$ for a virtual string link diagram $D$ (possibly $T$ itself).  
For each endpoint $e$ of $S$, consider a point $e'$ on $D$ which is adjacent to $e$ such that we meet $e$ and $e'$ consecutively in this order when going along orientation on $D$. 
Joining these new points by a copy of $S$, 
we can form a new subtree $S'$ such that $S$ and $S'$ run parallel and cross only at virtual crossings. 
Then $S$ and $S'$ are called two {\em parallel subtrees}.

The {\em expansion move} (E) for a ${\rm w}_{k}$-tree, having two variations, produces four {\rm w}-trees of degree $\leq k-1$ illustrated in Figure~\ref{expansion}. 
In the figure, the dotted lines on the left-hand side of ``$\stackrel{({\rm E})}{\longrightarrow}$'' 
represent two subtrees, which form the ${\rm w}_{k}$-tree together with represented part. 
The dotted parts on the right-hand side represent parallel copies of both subtrees.

\begin{figure}[htbp]
  \begin{center}
    \begin{overpic}[width=12cm]{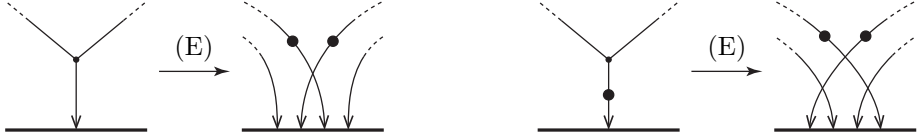}
      \put(63.5,29){(E)}
      \put(263,29){(E)}
    \end{overpic}
  \end{center}
  \caption{Expansion move}
  \label{expansion}
\end{figure}

Applying (E) recursively, we can turn any {\rm w}-tree into a union of {\rm w}-arrows. 
We call the union of {\rm w}-arrows 
the {\em expansion} of the {\rm w}-tree. 
The {\em surgery along a {\rm w}-tree} is surgery along its expansion. 
As before, $D_{T}$ denotes the result of surgery on $D$ along a union $T$ of {\rm w}-trees. 

As a natural generalization of arrow presentations, 
a {\em {\rm w}-tree presentation} for a virtual string link diagram $D$  is defined as a pair $(V,T)$ of a virtual string link diagram $V$ without classical crossings and a union $T$ of {\rm w}-trees for $V$ such that $V_{T}$ is welded isotopic to $D$.
Two {\rm w}-tree presentations $(V,T)$ and $(V',T')$ are {\em equivalent} if $V_{T}$ and $V'_{T'}$ are welded isotopic. 
Then arrow moves are extended to a set of local moves on {\rm w}-tree presentations, which are called {\em {\rm w}-tree moves}. 
It is proved that two {\rm w}-tree presentations are equivalent if and only if they are related by a sequence of {\rm w}-tree moves~\cite[Theorem 5.21]{MY}. 

In Section~\ref{sec-proofs}, 
we will use three kinds of {\rm w}-tree moves, 
{\em inverse}, {\em tails exchange} and {\em heads exchange moves} illustrated in Figure~\ref{wtree-moves}. 
The inverse move yields or deletes two parallel {\rm w}-trees which only differ by a twist on the terminal edge, 
the tails exchange move makes an exchange of two consecutive tails of {\rm w}-trees, and 
the heads exchange move makes an exchange of two consecutive heads of of {\rm w}-trees at the expense of an additional {\rm w}-tree illustrated in the lower right of Figure~\ref{wtree-moves}.

\begin{figure}[htbp]
  \begin{center}
    \begin{overpic}[width=12.5cm]{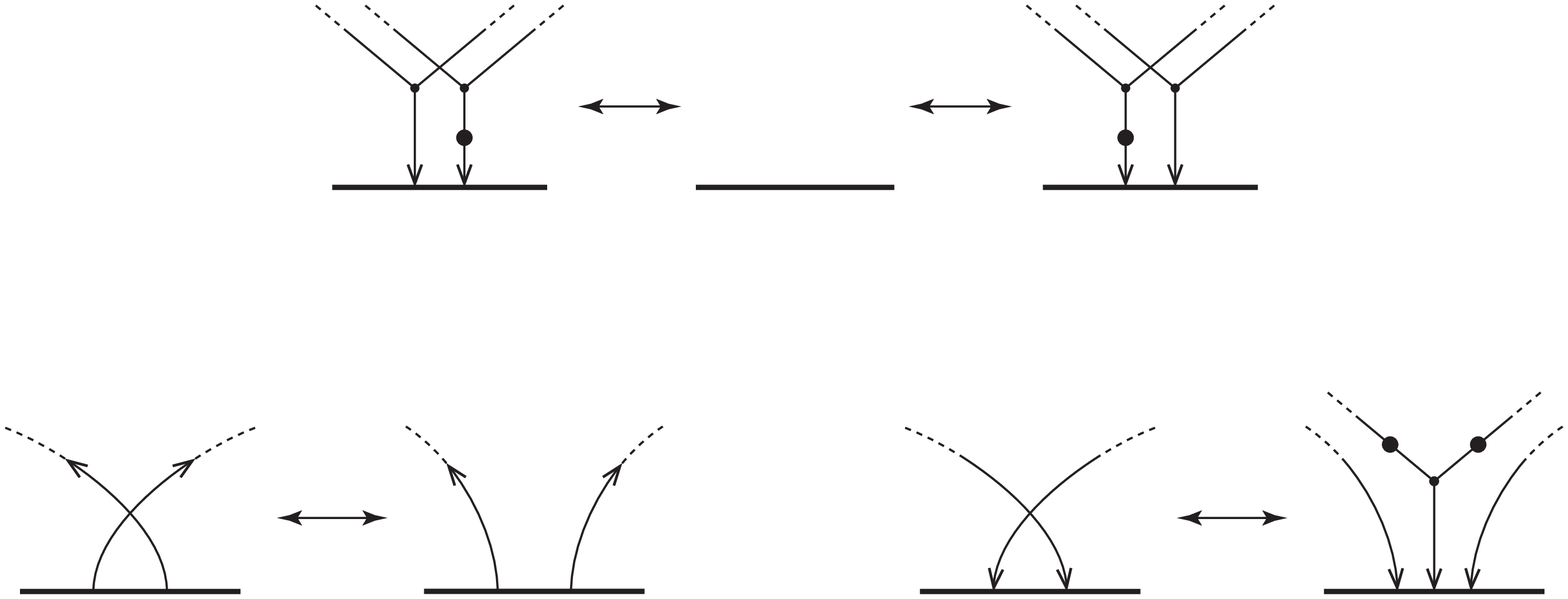}
      \put(151,76){inverse move}
      \put(33,-15){tails exchange move}
      \put(234,-15){heads exchange move}
    \end{overpic}
  \end{center}
  \vspace{1em}
  \caption{}
  \label{wtree-moves}
\end{figure}

\subsection{{\rm w}-tree moves up to {\rm sv}-equivalence}
\label{sec-HomotopyArrow}
To study virtual diagrams up to welded isotopy, 
we can work on {\rm w}-tree presentations together with {\rm w}-tree moves. 
Considering virtual diagrams up to {\rm sv}-equivalence, we can use some additional moves on {\rm w}-tree presentations. 
In this subsection, we recall these moves from~\cite[Section~9.1]{MY}. 
(Note that {\rm sv}-equivalence is called {\em homotopy} in~\cite{MY}.)

A {\em self-arrow} is a {\rm w}-arrow whose tail and head are attached to a single component of a virtual diagram. 
More generally, 
a {\em repeated {\rm w}-tree} is a {\rm w}-tree having two endpoints attached to a single component of a virtual diagram. 
Clearly, adding or deleting a self-arrow on {\rm w}-tree presentations corresponds to a self-crossing virtualization on virtual diagrams, 
i.e. surgery along a self-arrow does not change the {\rm sv}-equivalence class of a virtual diagram. 
This holds also for repeated {\rm w}-trees.

\begin{lemma}[{\cite[Lemma 9.2]{MY}}]\label{lem-repeated}
Surgery along a repeated {\rm w}-tree does not change the {\rm sv}-equivalence class of a virtual diagram. 
\end{lemma}

Exchanging a head and a tail of {\rm w}-trees of arbitrary degree can be achieved at the expense of an additional {\rm w}-tree as follows.

\begin{lemma}[{\cite[Lemma 9.3]{MY}}]\label{lem-HTexch}
Let $T_{1}$ be a ${\rm w}_{k}$-tree for a virtual diagram $D$, and let $T_{2}$ be a ${\rm w}_{l}$-tree for $D$.  
Let $T'_{1}\cup T'_{2}$ be obtained from $T_{1}\cup T_{2}$ by exchanging a tail of $T_{1}$ and the head of $T_{2}$. 
Then $D_{T_{1}\cup T_{2}}$ is {\rm sv}-equivalent to $D_{T'_{1}\cup T'_{2}\cup Y}$, 
where $Y$ denotes the ${\rm w}_{k+l}$-tree $T$ for $D$ illustrated in Figure~$\ref{HTexch}$. 
\end{lemma}

\begin{figure}[htbp]
  \begin{center}
    \begin{overpic}[width=6cm]{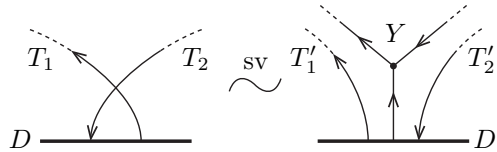}
      \put(81,28){sv}
      \put(-7,-2){$D$}
      \put(167,-2){$D$}
      \put(0,29){$T_{1}$}
      \put(58,29){$T_{2}$}
      \put(98.5,29){$T'_{1}$}
      \put(165,29){$T'_{2}$}
      \put(134,38){$Y$}
    \end{overpic}
  \end{center}
  \caption{Head-tail exchange move. Here, the notation ``$\stackrel{{\rm sv}}{\sim}$'' denotes that the virtual diagrams obtained by surgery along {\rm w}-trees are {\rm sv}-equivalent.}
  \label{HTexch}
\end{figure}

The modification of Figure~\ref{HTexch} is called a {\em head-tail exchange move}. 
Heads, tails and head-tail exchange moves are also referred to as {\em ends exchange moves}.

\section{Proofs}\label{sec-proofs}
In this section, we give the proofs of Theorems~\ref{th-2nsv} and~\ref{th-Vnsv}. 

Now we consider three local moves A, B and C on {\rm w}-tree presentations illustrated in Figures~\ref{arrow-2n} and~\ref{arrow-Vn}. 
Surgery along an A-move is equivalent to a $2n$-move whose strands are oriented parallel. 
On the other hand, surgery along a B-move is equivalent to a $V^{n}$-move. 
Furthermore, it is not hard to see that a B-move is equivalent to a C-move; see~\cite{MWY18}.

\begin{figure}[htbp]
  \begin{center}
    \begin{overpic}[width=3.5cm]{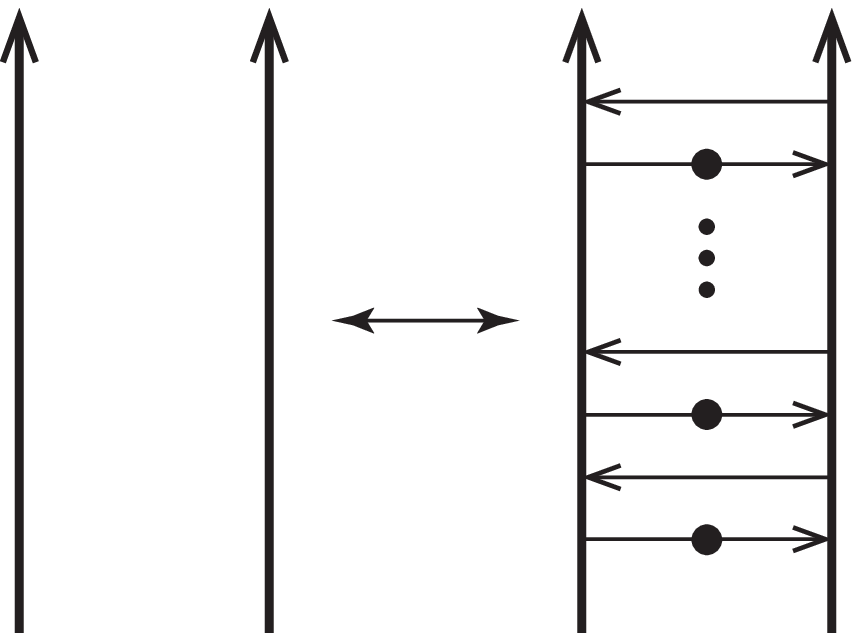}
      \put(46,40){A}
      \put(100,60){{\footnotesize $2n$}}
      \put(100,52){{\footnotesize $2n-1$}}
      \put(100,16){{\footnotesize $2$}}
      \put(100,8){{\footnotesize $1$}}
    \end{overpic}
  \end{center}
  \caption{}
  \label{arrow-2n}

\vspace{1em}

  \begin{center}
    \begin{overpic}[width=10cm]{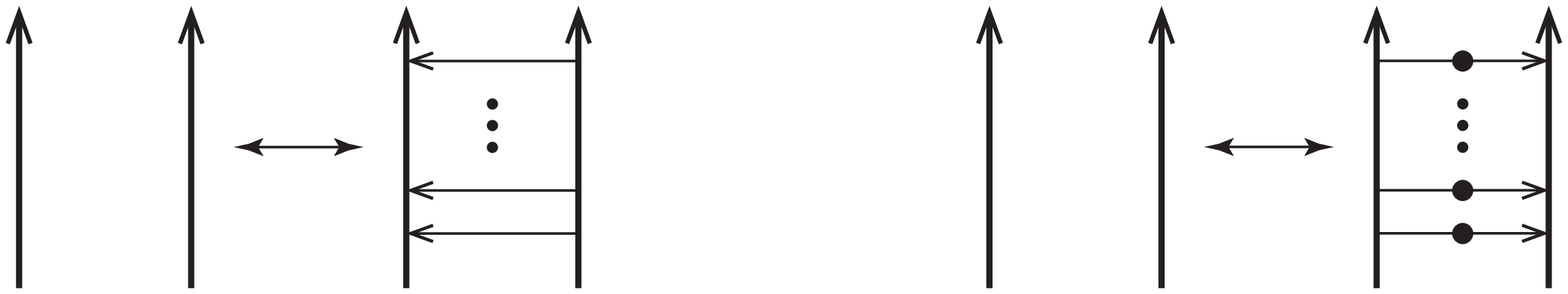}
      \put(50.5,29){B}
      \put(107,40){{\footnotesize $n$}}
      \put(107,15.5){{\footnotesize $2$}}
      \put(107,7){{\footnotesize $1$}}
      \put(227,29){C}
      \put(285,40){{\footnotesize $n$}}
      \put(285,15.5){{\footnotesize $2$}}
      \put(285,7){{\footnotesize $1$}}
    \end{overpic}
  \end{center}
  \caption{}
  \label{arrow-Vn}
\end{figure}

Using {\rm w}-tree presentations, we show the following.

\begin{proposition}\label{prop-2nVn}
Let $n$ be a positive integer. 
If two welded string links are $(2n+{\rm sv})$-equivalent, then they are $(V^{n}+{\rm sv})$-equivalent. 
\end{proposition}

\begin{proof}
A $2n$-move involving two strands of a single component is realized by link-homotopy\footnote{The equivalence relation on welded string links generated by the self-crossing change and welded isotopy is also referred to as link-homotopy.}. 
Furthermore, as seen in the proof of \cite[Theorem 3.1]{MWY19}, a $2n$-move whose two strands are oriented antiparallel is realized by link-homotopy and a $2n$-move whose strands are 
oriented parallel. 
Since link-homotopy implies {\rm sv}-equivalence, 
we may now assume that the orientations of the strands of a $2n$-move are always parallel. 

Up to $(V^{n}+{\rm sv})$-equivalence, by Lemmas~\ref{lem-repeated} and~\ref{lem-HTexch}, 
we can use {\rm w}-tree moves, B-, C-moves and ends exchange moves, and delete repeated {\rm w}-trees on {\rm w}-tree presentations. 
Hence, it is enough to show that an A-move is realized by a sequence of these operations, 
since a $2n$-move is realized by surgery along an A-move. 
Figure~\ref{2nVn-pf} indicates the proof. 
In the sequence of Figure~\ref{2nVn-pf} (a)--(c), we obtain (b) from (a) by B- and C-moves, and (c) from (b) by head-tail exchange moves and deleting repeated {\rm w}-trees. 
\end{proof}

\begin{figure}[htbp]
  \begin{center}
    \begin{overpic}[width=6cm]{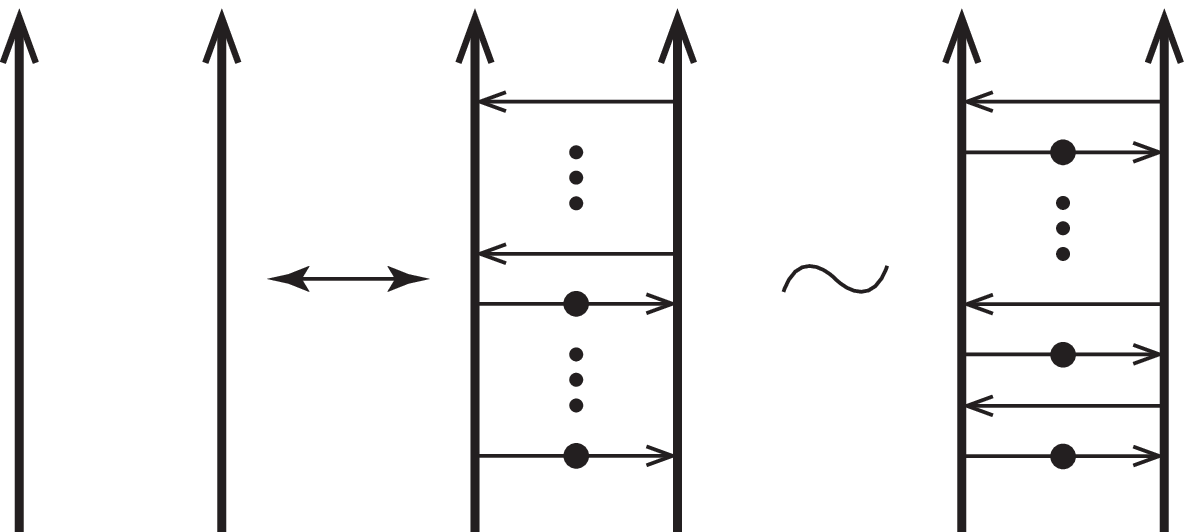}
      \put(40.5,41){B, C}
      \put(115.5,41){sv}
      \put(100,60){{\footnotesize $n$}}
      \put(100,38){{\footnotesize $1$}}
      \put(100,31){{\footnotesize $n$}}
      \put(100,8.5){{\footnotesize $1$}}
      \put(170,60){{\footnotesize $2n$}}
      \put(170,52){{\footnotesize $2n-1$}}
      \put(170,15.5){{\footnotesize $2$}}
      \put(170,8.5){{\footnotesize $1$}}
      \put(12,-14){(a)}
      \put(77,-14){(b)}
      \put(147,-14){(c)}
    \end{overpic}
  \end{center}
  \vspace{1em}
  \caption{}
  \label{2nVn-pf}
\end{figure}

For each integer $i\in\{1,\ldots,m\}$, let $\mathcal{S}_{k}(i)$ denote the set of all sequences $j_{1}\ldots j_{k}$ of $k$ distinct integers in $\{1,\ldots,m\}\setminus\{i\}$ such that $j_{r}<j_{k}$ for all $r$ $(=1,\ldots,k-1)$. 
For $I\in\mathcal{S}_{k}(i)$, 
let $T_{Ii}$ be the ${\rm w}_{k}$-tree for the trivial $m$-component string link diagram $\mathbf{1}_{m}$ illustrated in Figure~\ref{representative}, 
and let $\overline{T}_{Ii}$ be obtained from $T_{Ii}$ by inserting a twist in the terminal edge. 
Set 
$W_{Ii}=(\mathbf{1}_{m})_{T_{Ii}}$ and $W_{Ii}^{-1}=(\mathbf{1}_{m})_{\overline{T}_{Ii}}$. 
We remark that $W_{Ii}*W_{Ii}^{-1}$ is welded isotopic to $\mathbf{1}_{m}$ by applying an inverse move to $T_{Ii}\cup \overline{T}_{Ii}$, where the notation ``$*$'' denotes the stacking product. 
In~\cite{MY}, a complete list of representatives for welded string links up to {\rm sv}-equivalence was given in terms of {\rm w}-trees and welded Milnor invariants as follows.

\begin{figure}[htbp]
  \begin{center}
    \vspace{0.5em}
    \begin{overpic}[width=9.5cm]{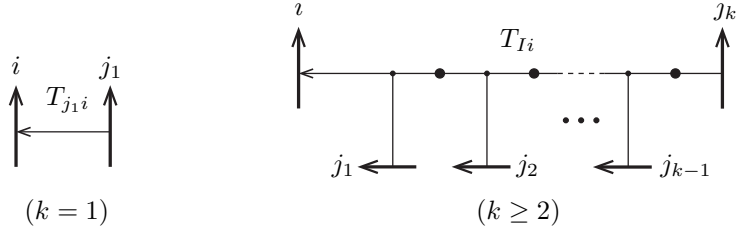}
      \put(1,38){$i$}
      \put(34,38){$j_{1}$}
      \put(6,-17){$(k=1)$}
      \put(14,26){$T_{j_{1}i}$}
      \put(184,48){$T_{Ii}$}
      \put(107,60){$i$}
      \put(264,60){$j_{k}$}
      \put(121,1){$j_{1}$}
      \put(190,1){$j_{2}$}
      \put(244,1){$j_{k-1}$}
      \put(175,-17){$(k\geq2)$}
    \end{overpic}
  \end{center}
  \vspace{1em}
  \caption{The ${\rm w}_{k}$-tree $T_{Ii}$ for $\mathbf{1}_{m}$}
  \label{representative}
\end{figure}

\begin{theorem}[{\cite[Theorem 9.4]{MY}}]
\label{th-rep-SV}
Let $\sigma$ be an $m$-component welded string link. 
Then $\sigma$ is {\rm sv}-equivalent to $\sigma_{1}*\cdots*\sigma_{m-1}$, where for each $k$, 
\[
\sigma_{k}=\prod_{i=1}^{m}\prod_{I\in\mathcal{S}_{k}(i)}
\left(W_{Ii}\right)^{x_{I}},~\text{with}~ 
x_{I}=\left\{
\begin{array}{lll}
\mu_{\sigma}^{{\rm w}}(ji) & (k=1,I=j), \\
\mu_{\sigma}^{{\rm w}}(Ii)-\mu_{\sigma_{1}*\cdots*\sigma_{k-1}}^{{\rm w}}(Ii) & (k\geq2).
\end{array}
\right.
\]
\end{theorem}

\vspace{0.5em}
The following plays an important role to show Theorems~\ref{th-2nsv} and~\ref{th-Vnsv}.

\begin{lemma}\label{lem-del-wtree}
Let $n$ be a positive integer and $\e\in\{1,-1\}$. 
Then, for any $I\in\mathcal{S}_{k}(i)$ and $k\geq2$, 
$\left(W_{Ii}\right)^{\e n}$ is $(2n+{\rm sv})$-equivalent to $\mathbf{1}_{m}$. 
\end{lemma}

\begin{proof}
Since $W_{Ii}^{-n}*W_{Ii}^{n}$ is welded isotopic to $\mathbf{1}_{m}$, it suffices to show the case $\e=1$, i.e.  
$\left(W_{Ii}\right)^{n}$ is $(2n+{\rm sv})$-equivalent to~$\mathbf{1}_{m}$ for a sequence $I=j_{1}\ldots j_{k}\in\mathcal{S}_{k}(i)$ $(k\geq2)$. 
Up to $(2n+{\rm sv})$-equivalence, 
we can use {\rm w}-tree moves, A-moves and ends exchange moves, and delete repeated {\rm w}-trees on {\rm w}-tree presentations. 
We relate $\mathbf{1}_{m}$ to a {\rm w}-tree presentation for $(W_{Ii})^{n}$ by a sequence of these operations. 

Figure~\ref{del-wtree}~(a)--(c) describes the intermediates in the sequence between $\mathbf{1}_{m}$ and a {\rm w}-tree presentation for $(W_{Ii})^{n}$. 
In the sequence, we obtain (a) from $\mathbf{1}_{m}$ by an inverse move, and (b) from (a) by an A-move. 
We obtain (c) from (b) by ends exchange moves on the $j_{k-1}$th component of $\mathbf{1}_{m}$ and deleting repeated {\rm w}-trees.
Finally, we obtain a {\rm w}-tree presentation for $(W_{Ii})^{n}$ from (c) by an inverse move and an A-move. 
\end{proof}

\begin{figure}[htbp]
  \begin{center}
    \vspace{0.5em}
    \begin{overpic}[width=9cm]{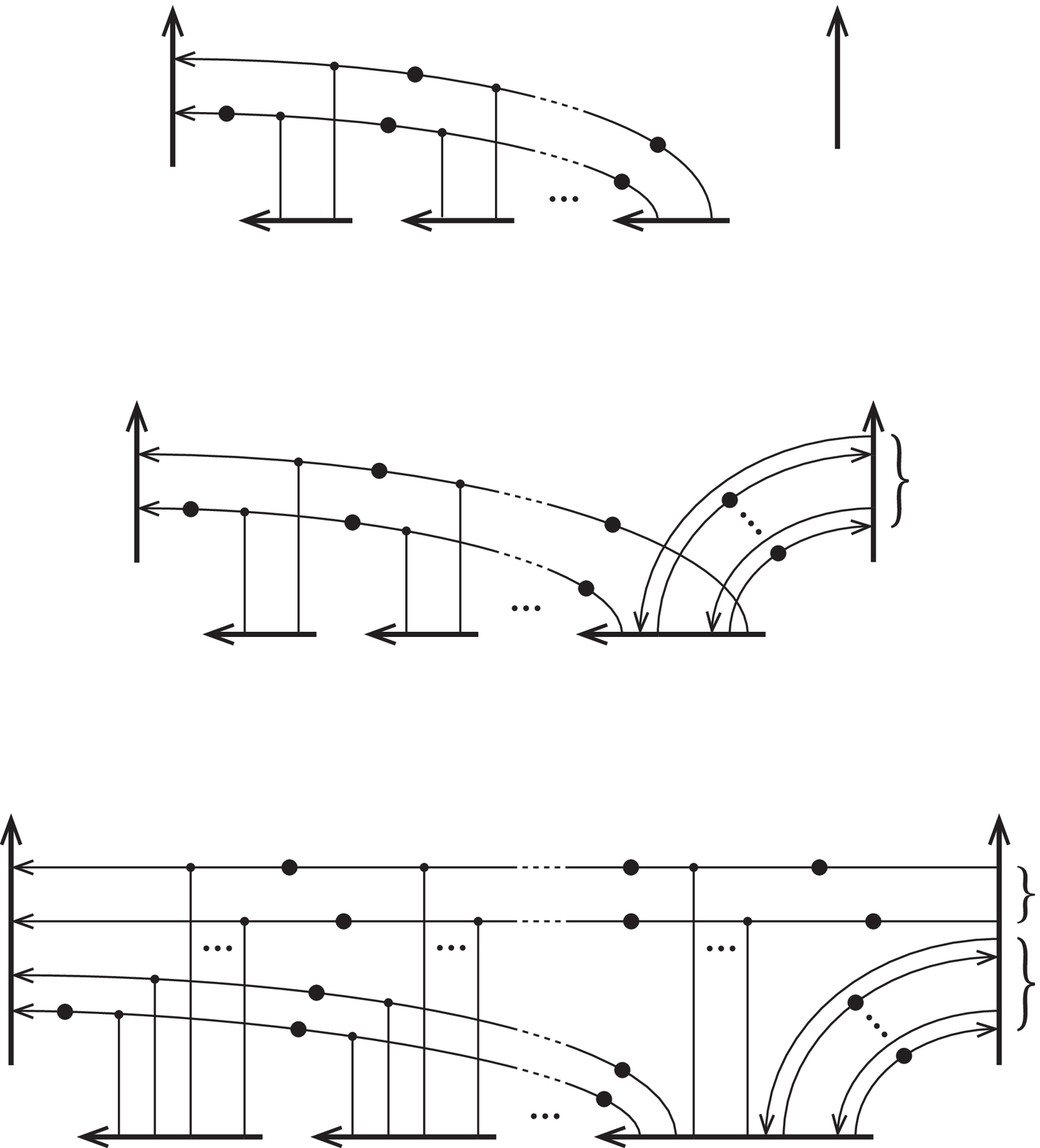}
      \put(119,206){(a)}
      \put(119,109){(b)}
      \put(119,-15){(c)}
      \put(41,287){{\small $i$}}
      \put(203,287){{\small $j_{k}$}}
       \put(53,222){{\small $j_{1}$}}
      \put(93,222){{\small $j_{2}$}}
      \put(181,222){{\small $j_{k-1}$}}
      \put(227,162){{\small $2n$}}
      \put(258,61){{\small $n$}}
      \put(258,38){{\small $2n$}}
    \end{overpic}
  \end{center}
  \vspace{1em}
  \caption{}
  \label{del-wtree}
\end{figure}

\begin{proposition}\label{prop-rep-2nSV}
Let $\sigma$ be an $m$-component welded string link, and let $x_{I}$ be as in Theorem~$\ref{th-rep-SV}$. 
Then $\sigma$ is $(2n+{\rm sv})$-equivalent to $\tau_{1}*\cdots*\tau_{m-1}$, where  
\[
\tau_{1}=\prod_{1\leq i<j\leq m}\left(W_{ji}^{y_{j}}*W_{ij}^{z_{i}}\right)
\]
for some $y_{j},z_{i}\in\mathbb{Z}$ with 
$0\leq y_{j}<n$ and $y_{j}\equiv x_{j}\pmod{n}$, 
and where for each $k\geq2$, 
\[
\tau_{k}=\prod_{i=1}^{m}\prod_{I\in\mathcal{S}_{k}(i)}
\left(W_{Ii}\right)^{y_{I}}
\]
with $0\leq y_{I}<n$ and $y_{I}\equiv x_{I}\pmod{n}$. 
\end{proposition}

\begin{proof}
It follows from Theorem~\ref{th-rep-SV} that $\sigma$ is {\rm sv}-equivalent to $\sigma_{1}*\cdots*\sigma_{m-1}$, where for each $k$ $(\geq1)$, 
\[
\sigma_{k}=\prod_{i=1}^{m}\prod_{I\in\mathcal{S}_{k}(i)}\left(W_{Ii}\right)^{x_{I}}. 
\]

For each $k\geq2$, by Lemma~\ref{lem-del-wtree} and the fact that $W_{Ii}^{\e}*W_{Ii}^{-\e}$ is welded isotopic to $\mathbf{1}_{m}$ $(\e\in\{1,-1\})$, 
$\sigma_{k}$ is $(2n+{\rm sv})$-equivalent to $\tau_{k}$.

Now consider the case $k=1$. 
Performing ends exchange moves and deleting repeated {\rm w}-trees, 
a {\rm w}-tree presentation for $\sigma_{1}$ is deformed into 
one for $\sigma'_{1}$, where 
\[
\sigma'_{1}=\prod_{1\leq i<j\leq m}\left(W_{ji}^{x_{j}}*W_{ij}^{x_{i}}\right).
\]
By Lemmas~\ref{lem-repeated} and~\ref{lem-HTexch}, $\sigma_{1}$ is {\rm sv}-equivalent to $\sigma'_{1}$. 
Furthermore, 
a {\rm w}-tree presentation for $\sigma'_{1}$ is deformed into one for $\tau_{1}$ by using A-moves, ends exchange moves and inverse moves. 
Therefore, $\sigma'_{1}$ is $(2n+{\rm sv})$-equivalent to $\tau_{1}$. 
This completes the proof. 
\end{proof}

\begin{proposition}\label{prop-rep-VnSV}
Let $\sigma$ be an $m$-component welded string link, and let $x_{I}$ be as in Theorem~$\ref{th-rep-SV}$. 
Then $\sigma$ is $(V^{n}+{\rm sv})$-equivalent to $\tau_{1}*\cdots*\tau_{m-1}$, where for each $k$, 
\[
\tau_{k}=\prod_{i=1}^{m}\prod_{I\in\mathcal{S}_{k}(i)}
\left(W_{Ii}\right)^{y_{I}}
\]
with $0\leq y_{I}<n$ and $y_{I}\equiv x_{I}\pmod{n}$. 
\end{proposition}

\begin{proof}
For any $j\in\mathcal{S}_{1}(i)$, we see that 
$(W_{ji})^{\pm n}$ and $\mathbf{1}_{m}$ are related by a single $V^{n}$-move. 
This together with Proposition~\ref{prop-2nVn} and Lemma~\ref{lem-del-wtree} implies that  
$(W_{Ii})^{\pm n}$ is $(V^{n}+{\rm sv})$-equivalent to $\mathbf{1}_{m}$ for any $I\in\mathcal{S}_{k}(i)$ and $k\geq1$.
Therefore, the proof can be done by arguments similar to those in the proof of Proposition~\ref{prop-rep-2nSV}. 
\end{proof}

\begin{proof}[Proof of Theorem~$\ref{th-2nsv}$]
This follows from Propositions~\ref{prop-inv-2n} and~\ref{prop-rep-2nSV}. 
\end{proof}

\begin{proof}[Proof of Theorem~$\ref{th-Vnsv}$]
This follows from Theorem~\ref{th-inv-Vn} and Proposition~\ref{prop-rep-VnSV}. 
\end{proof}

\begin{remark}\label{rem-representative}
By Theorem~\ref{th-2nsv} (resp. Theorem~\ref{th-Vnsv}), 
we can conclude that Proposition~\ref{prop-rep-2nSV} (resp. Proposition~\ref{prop-rep-VnSV}) 
gives a complete list of representatives for welded string links up to $(2n+{\rm sv})$-equivalence (resp. $(V^{n}+{\rm sv})$-equivalence). 
\end{remark}

\begin{proof}[Proof of Corollary~$\ref{cor-group}$]
This follows from Remark~\ref{rem-representative} and Lemma~\ref{lem-del-wtree}. 
\end{proof}

In the rest of this paper, we discuss 
the relations between $(2n+{\rm lh})$-, $(2n+{\rm sv})$- and $(V^{n}+{\rm sv})$-equivalence. 

It is not hard to see that $(2n+{\rm lh})$-equivalence implies $(2n+{\rm sv})$-equivalence. 
Furthermore, $(2n+{\rm sv})$-equivalence implies $(V^{n}+{\rm sv})$-equivalence by Proposition~\ref{prop-2nVn}. 
The converse implications hold for classical string links as follows.

\begin{proposition}\label{prop-injective}
Let $n$ be a positive integer, and let $\sigma$ and $\sigma'$ be classical string links. 
The following assertions $(1)$, $(2)$ and $(3)$ are equivalent:  
\begin{enumerate}
\item $\sigma$ and $\sigma'$ are $(2n+\rm{lh})$-equivalent.

\item $\sigma$ and $\sigma'$ are $(2n+\rm{sv})$-equivalent.

\item $\sigma$ and $\sigma'$ are $(V^{n}+\rm{sv})$-equivalent. 
\end{enumerate}
\end{proposition}

\begin{proof}
It is enough to show the implications $(2)\Rightarrow(1)$ and $(3)\Rightarrow(1)$. 

The proof of the implication $(2)\Rightarrow(1)$ is given as follows. 
If $\sigma$ and $\sigma'$ are $(2n+\rm{sv})$-equivalent, 
then $\mu^{{\rm w}}_{\sigma}(I)\equiv\mu^{{\rm w}}_{\sigma'}(I)\pmod{n}$ for any non-repeated sequence $I$ by Theorem~\ref{th-2nsv}. 
It follows from Remark~\ref{rem-weldedMilnor} that 
$\mu^{{\rm w}}_{\sigma}(I)=\mu_{\sigma}(I)$ and $\mu^{{\rm w}}_{\sigma'}(I)=\mu_{\sigma'}(I)$. 
Hence, Theorem~\ref{th-classical} completes the proof. 

Using Theorem~\ref{th-Vnsv} instead of Theorem~\ref{th-2nsv}, the implication $(3)\Rightarrow(1)$ is similarly shown. 
\end{proof}



\end{document}